\documentclass{gLMA2e}

\usepackage{epstopdf}
\usepackage{subfigure}
\usepackage[ruled,vlined,linesnumbered]{algorithm2e}
\usepackage[colorlinks=true, pdfstartview=FitV, linkcolor=blue, citecolor=blue, urlcolor=blue]{hyperref}

\theoremstyle{plain}
\newtheorem{theorem}{Theorem}[section]
\newtheorem{corollary}[theorem]{Corollary}
\newtheorem{lemma}[theorem]{Lemma}
\newtheorem{proposition}[theorem]{Proposition}

\theoremstyle{definition}
\newtheorem{definition}{Definition}

\theoremstyle{remark}
\newtheorem{remark}{Remark}

\newcommand{\Stief}{\mathrm{St}}

\usepackage[normalem]{ulem} 

\renewcommand{\bm}[1]{\boldsymbol{#1}}

\newcommand{\mbfa}{\mathbf{A}}

\newcommand{\mbfy}{\mathbf{Y}}

\newcommand{\mbr}{\mathbb{R}}

\newcommand{\mcc}{\mathcal{C}}

\newcommand{\mcx}{\mathcal{X}}
\newcommand{\mcy}{\mathcal{Y}}

\newcommand{\vg}{{\mathbf{g}}}

\newcommand{\vx}{{\mathbf{x}}}

\newcommand{\vA}{{\mathbf{A}}}

\newcommand{\vG}{{\mathbf{G}}}

\newcommand{\vI}{{\mathbf{I}}}

\newcommand{\vU}{{\mathbf{U}}}
\newcommand{\vV}{{\mathbf{V}}}
\newcommand{\vW}{{\mathbf{W}}}
\newcommand{\vX}{{\mathbf{X}}}
\newcommand{\vY}{{\mathbf{Y}}}
\newcommand{\vZ}{{\mathbf{Z}}}

\newcommand{\cB}{{\mathcal{B}}}
\newcommand{\cC}{{\mathcal{C}}}

\newcommand{\cH}{{\mathcal{H}}}

\newcommand{\cK}{{\mathcal{K}}}
\newcommand{\cL}{{\mathcal{L}}}

\newcommand{\cN}{{\mathcal{N}}}
\newcommand{\cO}{{\mathcal{O}}}

\newcommand{\cS}{{\mathcal{S}}}

\newcommand{\cU}{{\mathcal{U}}}

\newcommand{\cX}{{\mathcal{X}}}
\newcommand{\cY}{{\mathcal{Y}}}

\newcommand{\RR}{\mathbb{R}} 
\newcommand{\vzero}{\mathbf{0}} 

\newcommand{\dist}{\mathrm{dist}}    

\newcommand{\unfold}{{\mathbf{unfold}}} 

\newcommand{\st}{\mbox{ s.t. }}

\DeclareMathOperator*{\argmin}{arg\,min} 
\DeclareMathOperator*{\argmax}{arg\,max} 

\newcommand{\bc}{\begin{center}}
\newcommand{\ec}{\end{center}}

\newcommand{\bdm}{\begin{displaymath}}
\newcommand{\edm}{\end{displaymath}}

\newcommand{\beq}{\begin{equation}}
\newcommand{\eeq}{\end{equation}}

\newcommand{\bfl}{\begin{flushleft}}
\newcommand{\efl}{\end{flushleft}}

\newcommand{\bt}{\begin{tabbing}}
\newcommand{\et}{\end{tabbing}}

\newcommand{\beqn}{\begin{eqnarray}}
\newcommand{\eeqn}{\end{eqnarray}}

\newcommand{\beqs}{\begin{align*}} 
\newcommand{\eeqs}{\end{align*}}  

\begin{document}



\title{On the convergence of higher-order orthogonal iteration\thanks{This work is partly supported by NSF grant DMS-1719549.}}

\author{
\name{Yangyang Xu\textsuperscript{a}$^{\ast}$\thanks{$^\ast$Corresponding author. Email: xuy21@rpi.edu}}
\affil{\textsuperscript{a}Department of Mathematical Sciences\\ Rensselaer Polytechnic Institute\\
110 8th Street, Troy, NY 12180, USA}
}

\maketitle

\begin{abstract}
The higher-order orthogonal iteration (HOOI) has been popularly used for finding a best low-multilinear-rank approximation of a tensor. However, its convergence is still an open question. In this paper, we first analyze a greedy HOOI, which updates each factor matrix by selecting from the best candidates one that is closest to the current iterate. Assuming the existence of a block-nondegenerate cluster point, we establish its global iterate sequence convergence through the so-called Kurdyka-{\L}ojasiewicz (KL) property. In addition, we show that if the starting point is sufficiently close to any block-nondegenerate globally optimal solution, the greedy HOOI produces an iterate sequence convergent to a globally optimal solution. Relating the iterate sequence by the original HOOI to that by the greedy HOOI, we then show that the original HOOI has global convergence on the multilinear subspace sequence and thus positively address the open question.
\end{abstract}


\begin{classcode}9008; 90C26; 90C59;\end{classcode}

\section{Introduction}
It is shown in \cite{de2000multilinear} that any tensor (i.e., multi-dimensional array) can be decomposed into the product of orthogonal matrices and an \emph{all-orthogonal} core tensor. This decomposition generalizes the matrix SVD and is today commonly called higher-order singular value decomposition (HOSVD) or multilinear SVD. In applications, {a low-multilinear-rank approximation of a given tensor is commonly used}, such as the multilinear subspace learning \cite{lu2011survey}, multilinear principal component analysis \cite{lu2008mpca}, tensor decomposition in signal processing \cite{cichocki2015tensor}, just to name a few. Unlike the matrix SVD, \emph{truncated} HOSVD can give a good but not necessarily the best low-multilinear-rank approximation of the given tensor. To obtain a better approximation, existing works (e.g., \cite{kroonenberg1980principal, de2000best, elden2009newton}) solve the best rank-$(r_1,\ldots,r_N)$ approximation problem
\begin{equation}\label{eq:hosvd}
\min_{\bm{\cC},\vA}\|\bm{\cX}-\bm{\cC}\times_1\vA_1\ldots\times_N\vA_N\|_F^2,\st \vA_n\in\Stief_{I_n\times r_n},\,\forall n,
\end{equation}
where $\bm{\cX}\in\RR^{I_1\times\ldots\times I_N}$ is a given tensor, $\times_n$ denotes mode-$n$ tensor-matrix multiplication (see the definition in \eqref{eq:tm} below), and
$$\Stief_{I_n\times r_n}=\{\vA_n\in\RR^{I_n\times r_n}: \vA_n^\top\vA_n=\vI\}$$
is the Stiefel manifold. 
With $\vA$ fixed, the optimal core tensor is given by $\bm{\cC}=\bm{\cX}\times_1\vA_1^\top\ldots\times_N\vA_N^\top$. Absorbing this $\bm{\cC}$ into the objective, one can write \eqref{eq:hosvd} equivalently to (see \cite[Theorem 3.1]{de2000best} for detailed derivation)
\begin{equation}\label{eq:hosvd2}
\max_{\vA}\|\bm{\cX}\times_1\vA_1^\top\ldots\times_N\vA_N^\top\|_F^2,\st \vA_n\in\Stief_{I_n\times r_n},\,\forall n.
\end{equation}

One popular method for solving \eqref{eq:hosvd2} is the higher-order orthogonal iteration (HOOI) (see Algorithm \ref{alg:hooi}). HOOI is commonly used and practically efficient (already coded in the Matlab Tensor Toolbox \cite{TTB_Software-2012} and Tensorlab \cite{Tensorlab-2014}). Its iterate sequence convergence has been established for the case of rank-one tensor decomposition \cite{wang2014global, uschmajew2015new}. However, for general cases, existing works only show that the objective value of \eqref{eq:hosvd2} at the generated iterates increasingly converges to some value while the iterate sequence convergence is still an open question (c.f. \cite{de2000multilinear}). In this paper, we address this open question by showing a result on \emph{multilinear subspace convergence}. This result is important because without convergence, running the algorithm to different numbers of iterations may give severely different multilinear subspaces, and that will ultimately affect the results of applications.  Our main results are summarized in the following theorem.

\begin{theorem}[Main Theorem]\label{thm:main}
Let $\{\vA^k\}_{k\ge1}$ be the sequence generated by the HOOI method. We have:

\emph{(i)}. If $\{\vA^k\}_{k\ge1}$ has a block-nondegenerate (see Definition \ref{def:nondeg}) cluster point $\bar{\vA}$, then $\bar{\vA}$ is a critical point and also a block-wise maximizer of \eqref{eq:hosvd2}. In addition, $\underset{k\to\infty}\lim \vA^k(\vA^k)^\top=\bar{\vA}\bar{\vA}^\top$, where $$\vA\vA^\top=(\vA_1\vA_1^\top,\ldots,\vA_N\vA_N^\top).$$ 

\emph{(ii)}. If the starting point $\vA^0$ is sufficiently close to any block-nondegenerate local maximizer of \eqref{eq:hosvd2}, then the entire sequence $\{\vA^k(\vA^k)^\top\}_{k\ge1}$ must converge to some point $\bar{\vA}\bar{\vA}^\top$ and $\bar{\vA}$ is a local maximizer of \eqref{eq:hosvd2}.
\end{theorem}

We make some remarks on the assumption and the convergence results.
\begin{remark}\label{rm:main} 
The block-nondegeneracy assumption is also necessary because even starting from a critical point $\bar{\vA}$, the HOOI method can still deviate from $\bar{\vA}$ if it is not block-nondegenerate (see Remark \ref{rm:non-deg}), that is, a degenerate critical point is not stable (see \cite{elden2011perturbation} for the perturbation analysis). In practice, the block-nondegeneracy is always observed\footnote{Here, we assume $r_n\le \prod_{i\neq n}r_i,\,\forall n$. If otherwise, for some $n$, $r_n> \prod_{i\neq n}r_i$, we can reduce $r_n$ to $\prod_{i\neq n}r_i$ without changing the approximation in \eqref{eq:hosvd} because $\vG_n^k\in\RR^{I_n\times \prod_{i\neq n}r_i}$.}, and it is implied by $\liminf_k \big(\sigma_{r_n}(\vG_n^k)-\sigma_{r_n+1}(\vG_n^k)\big)>0,\,\forall n$, where $\vG_n^k$ is defined in \eqref{eq:gnk}; see Figure \ref{fig:svGap}. 

The assumption is similar to the one assumed by the orthogonal iteration method \cite[section 7.3.2]{GolubVanLoan1996} for computing $r$-dimensional dominant invariant subspace of a matrix $\vX$. Typically, the convergence of the orthogonal iteration method requires that there is a positive gap between the $r$-th and $(r+1)$-th largest eigenvalues of $\vX$ in magnitude, because otherwise, the $r$-dimensional dominant invariant subspace of $\vX$ is not unique.

For a block-wise maximizer $\bar{\vA}$, its block-nondegeneracy is equivalent to negative definiteness of each block Hessian over the Stiefel manifold $\Stief_{I_n\times r_n}$. The definition of our block-nondegeneracy is different from the nondegeneracy in \cite{ishteva2011best}. A nondegenerate local maximizer in \cite{ishteva2011best} is one local maximizer that has negative definite Hessian, so the nondegeneracy assumption in \cite{ishteva2011best} is strictly stronger than our block-nondegeneracy assumption. 
\end{remark}

\begin{remark}
Since the solution to each subproblem (see \eqref{eq:a-sub}) of the HOOI method is not unique and actually still a solution after multiplying any orthogonal matrix to its right, we can only hope to establish convergence of the projection matrix sequence $\{\vA^k(\vA^k)^\top\}_{k\ge1}$ instead of $\{\vA^k\}_{k\ge1}$ itself. Hence, the convergence result is on the product manifold of $\Stief_{I_n\times r_n}$ and similar to that in  \cite{ishteva2011best}. However, our assumption is strictly weaker, and thus different tools are used.
\end{remark}

\subsection{Basic concepts of tensor} Before proceeding with discussion, we first review some basic concepts about tensor that we use in this paper. 

The $(i_1,\ldots,i_N)$-th component of an $N$-way tensor $\bm{\mcx}$ is denoted as $x_{i_1\ldots i_N}$. For $\bm{\mcx},\bm{\mcy}\in\mbr^{m_1\times\ldots\times m_N}$, their inner product is defined in the same way as that for matrices, i.e.,
$$\langle\bm{\mcx},\bm{\mcy}\rangle=\sum_{i_1=1}^{m_1}\cdots\sum_{i_N=1}^{m_N}x_{i_1\ldots i_N}\cdot y_{i_1\ldots i_N}.$$
The Frobenius norm of $\bm{\mcx}$ is defined as $\|\bm{\mcx}\|_F=\sqrt{\langle\bm{\mcx},\bm{\mcx}\rangle}.$ A \emph{fiber} of $\bm{\mcx}$ is a vector obtained by fixing all indices of $\bm{\mcx}$ except one. The mode-$n$ \emph{matricization} (also called \emph{unfolding}) of  $\bm{\mcx}$ is denoted as $\unfold_n(\bm{\cX})$, which is a matrix with columns being the mode-$n$ fibers of $\bm{\mcx}$ in the lexicographical order. The mode-$n$ product of $\bm{\mcx}\in\mbr^{m_1\times\cdots\times m_N}$ with $\vY\in\mbr^{p\times m_n}$ is written as $\bm{\mcx}\times_n\mbfy$ which gives a tensor in $\mbr^{m_1\times \cdots \times m_{n-1}\times p\times m_{n+1}\times \cdots\times m_N}$ and is defined component-wisely by
\begin{equation}\label{eq:tm}
(\bm{\mcx}\times_n \mbfy)_{i_1\cdots i_{n-1}ji_{n+1}\cdots i_N}=\sum_{i_n=1}^{m_n}x_{i_1i_2\cdots i_N}\cdot y_{ji_n}.
\end{equation}
If $\bm{\mcx}=\bm{\mcc}\times_1\mbfa_1\ldots\times_N\vA_N$, then for any $n$,
\begin{align}\label{eq:mat}\unfold_n(\bm{\cX})=&\,\mbfa_n\unfold_n(\bm{\cC})(\vA_N\otimes\ldots\otimes\vA_{n+1}\otimes\vA_{n-1}\otimes\ldots\otimes\vA_1)^\top,\cr
=&\,\vA_n\unfold_n(\bm{\cC}\times_1\vA_1\ldots\times_{n-1}\vA_{n-1}
\times_{n+1}\vA_{n+1}\ldots\times_N\vA_N),
\end{align}
where ``$\otimes$'' denotes the Kronecker product.

\subsection{Higher-order orthogonal iteration}
The HOOI method updates $\vA$ by maximizing the objective of \eqref{eq:hosvd2} alternatingly with respect to $\vA_1, \vA_2,\ldots,\vA_N$, one factor matrix at a time while the remaining ones are fixed. Specifically, assuming the iterate to be $\vA^k$ at the beginning of the $k$-th iteration, it performs the following update sequentially from $n=1$ through $N$:
\begin{equation}\label{eq:a-sub}
\vA^{k+1}_n\in\argmax_{\vA_n\in\Stief_{I_n\times r_n}} \|\vA_n^\top\vG_n^k\|_F^2,
\end{equation}
where we have used \eqref{eq:mat}, and
\begin{equation}\label{eq:gnk}
\vG_n^k=\unfold_n(\bm{\cX}\times_{i < n}(\vA_i^{k+1})^\top\times_{i>n}(\vA_i^k)^\top).
\end{equation}
Any orthonormal basis of the dominant $r_n$-dimensional left  singular subspace of $\vG_n^k$ is a solution of \eqref{eq:a-sub}. The pseudocode of HOOI is given in Algorithm \ref{alg:hooi}.

\begin{algorithm}\caption{Higher-order orthogonal iteration (HOOI)}\label{alg:hooi}
{\small
\DontPrintSemicolon
\textbf{Input:} $\bm{\cX}$ and $(r_1,\ldots,r_N)$\;
\textbf{Initialization:} choose $(\vA_1^0,\ldots,\vA_N^0)$ with $\vA_n^0\in\Stief_{I_n\times r_n},\,\forall n$\;
\For{$k=0,\ldots,$}{
\For{$n=1,\ldots,N$}{
Set $\vA_n^{k+1}$ to an orthonormal basis of the dominant $r_n$-dimensional left  singular subspace of $\vG_n^k$.
}
\If{Some stopping criteria are met}{
Output $\vA=\vA^{k+1},\, \bm{\cC}=\bm{\cX}\times_1\vA_1\ldots\times_N\vA_N$ and stop.
}
}
}
\end{algorithm}

It is easy to implement Algorithm \ref{alg:hooi} by simply setting $\vA_n^{k+1}$ to the left $r_n$ leading singular vectors of $\vG_n^k$. This implementation is adopted in the Matlab Tensor Toolbox \cite{TTB_Software-2012} and Tensorlab \cite{Tensorlab-2014}. 
However, we did not find any work that gives a convergence result of HOOI, except for our recent paper \cite{xu2014higher} that establishes subsequence convergence by assuming a strong condition on the entire iterate sequence. The essential difficulty is the non-uniqueness of the solution of \eqref{eq:a-sub}, and the leading singular vectors are not uniquely determined either. 

To tackle this difficulty, we first analyze a greedy method, which always chooses one solution of \eqref{eq:a-sub} that is closest to $\vA_n^k$ as follows:
\begin{equation}\label{eq:ghooi}
\vA_n^{k+1}\in\argmin_{\vA_n\in\cH_n^k} \|\vA_n-\vA_n^k\|_F^2,
\end{equation}
where
\begin{equation}\label{eq:hnk}
\cH_n^k=\argmax_{\vA_n\in\Stief_{I_n\times r_n}}\|\vA_n^\top\vG_n^k\|_F^2.
\end{equation}
The pseudocode of the greedy implementation is shown in Algorithm \ref{alg:ghooi}. The subproblem in \eqref{eq:ghooi} can be solved by the method given in Remark \ref{rm:sol-sub}.
Although \eqref{eq:ghooi} can in general have multiple solutions, we will show that near any cluster point of the iterate sequence, it must have a unique solution. With the greedy implementation, we are able to establish iterate sequence convergence of the greedy HOOI method (i.e., Algorithm \ref{alg:ghooi}), as shown in section  \ref{sec:analysis}. Through relating (see \eqref{equal-iter} and Figure \ref{fig:cvg-bh}) the two iterate sequences generated by the original (i.e., Algorithm \ref{alg:hooi}) and greedy HOOI methods, we then establish the multilinear subspace convergence of the original HOOI method, as shown in section \ref{sec:pf-of-main}. 
\begin{algorithm}\caption{Greedy higher-order orthogonal iteration (Greedy-HOOI)}\label{alg:ghooi}
{\small
\DontPrintSemicolon
\textbf{Input:} $\bm{\cX}\in\RR^{I_1\times\ldots\times I_N}$ and $(r_1,\ldots,r_N)$\;
\textbf{Initialization:} choose $(\vA_1^0,\ldots,\vA_N^0)$ with $\vA_n^0\in\Stief_{I_n\times r_n},\,\forall n$\;
\For{$k=0,\ldots,$}{
\For{$n=1,\ldots,N$}{
Set $\vA_n^{k+1}$ by \eqref{eq:ghooi}
}
\If{Some stopping criteria are met}{
Output $\vA=\vA^{k+1},\, \bm{\cC}=\bm{\cX}\times_1\vA_1\ldots\times_N\vA_N$ and stop.
}
}
}
\end{algorithm}

\subsection{Comparison to other methods}
Besides the HOOI method, several other methods have been developed for solving the low-multilinear-rank tensor approximation problem. One of the earliest methods, called TUCKALS3, was proposed in \cite{kroonenberg1980principal}. TUCKALS3 also sequentially updates $\vA_1$ through $\vA_N$ and then cycles the process, but different from HOOI, it obtains approximate leading left singular vectors of $\vG_n^k$ by carrying out only one step of the so-called Bauer-Rutishauser method \cite{rutishauser1969computational} starting from $\vA_n^k$. This update is equivalent to solving a linearized version of the subproblem \eqref{eq:a-sub}, and it prevents $\vA_n^{k+1}$ being far away from $\vA^k_n$. Subsequence convergence of TUCKALS3 was established under the assumption that $(\vA_n^k)^\top\vG_n^k(\vG_n^k)^\top\vA_n^k$ is positive definite for all $n$ and $k$. Although TUCKALS3 has slightly lower per-iteration complexity than HOOI, it does not converge as fast as HOOI as demonstrated in Figure \ref{fig:cvg-bh}.

Recently, some Newton-type methods on manifolds were developed for the low-multilinear-rank tensor approximation problem such as the (quasi)Newton-Grassmann method \cite{elden2009newton, savas2010quasi}, geometric Newton method \cite{ishteva2009differential} and the Riemannian trust region scheme \cite{ishteva2011best}. These methods usually exhibit superlinear convergence. Numerical experiments in \cite{ishteva2011best} demonstrate that for small-size problems, the Riemannian trust region scheme and/or Newton-type methods can take much fewer iterations and also less time than the HOOI method to reach a high-level accuracy based on violation to the first-order optimality condition. However, for medium-size or large-size problems, or if only medium-level accuracy is required, the HOOI method is superior over the Riemannian trust region scheme and also several other Newton-type methods.

Under negative definiteness assumption on the Hessian of a local maximizer, the Newton-type methods are guaranteed to have superlinear or even quadratic local convergence (c.f. \cite{ishteva2011best}). Compared to our block-nondegeneracy assumption, their assumption is strictly stronger because as mentioned in Remark \ref{rm:main}, for a local maximizer, its block-nondegeneracy is equivalent to the negative definiteness of each block Hessian. Only with block-nondegeneracy assumption, it is not clear how to show the local convergence of the Newton-type methods. 

\begin{figure}
\centering
\begin{tabular}{cc}
Synthetic data & Yale Face Database B\\
\includegraphics[width=0.35\textwidth]{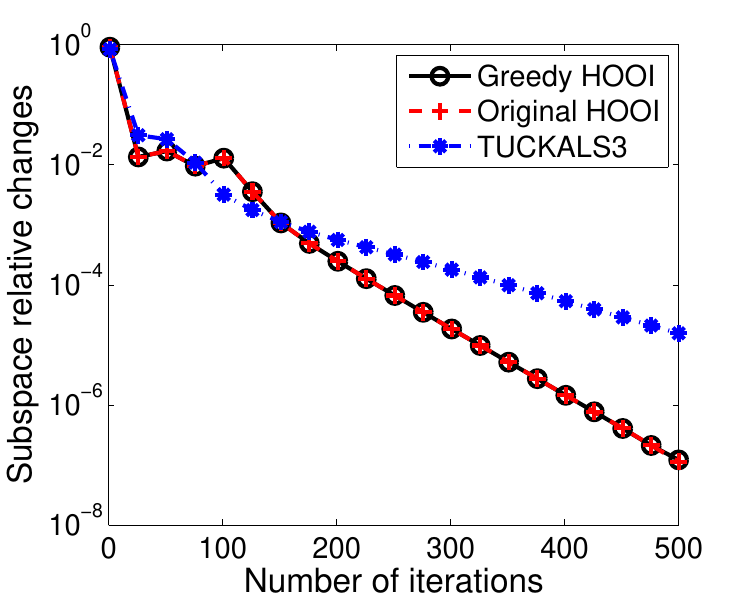} &
\includegraphics[width=0.35\textwidth]{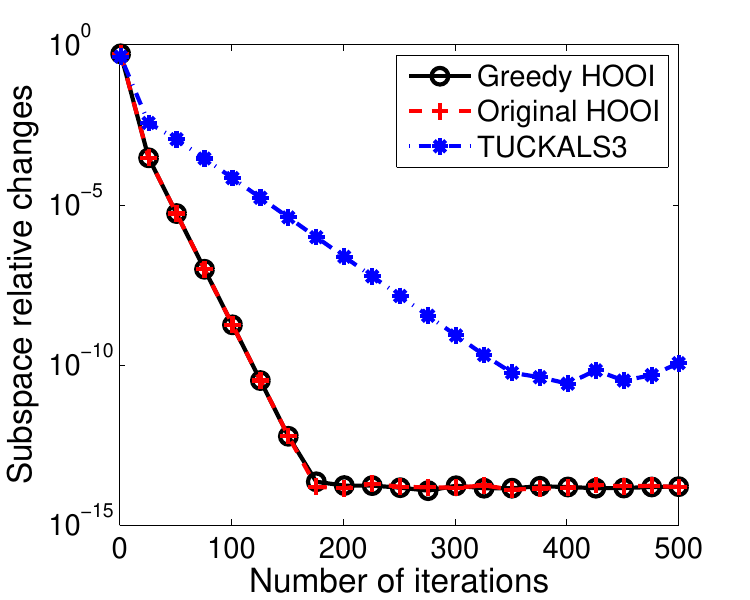}\\
\includegraphics[width=0.35\textwidth]{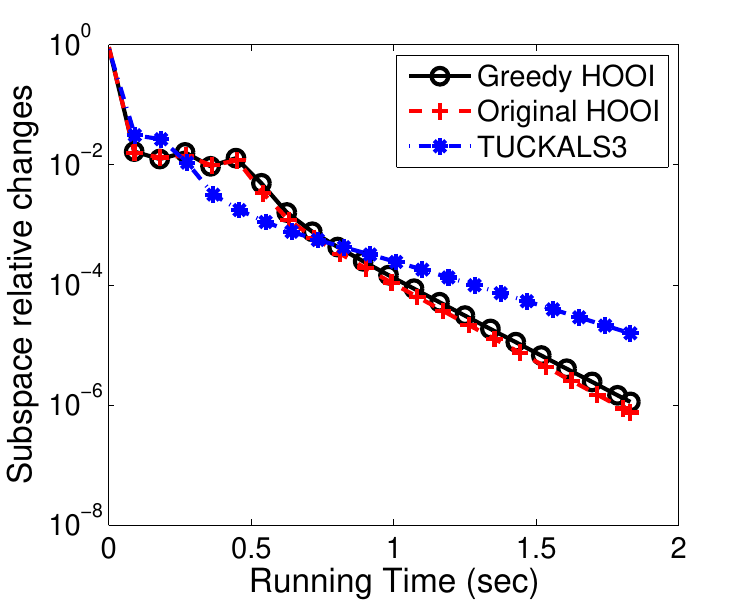} &
\includegraphics[width=0.35\textwidth]{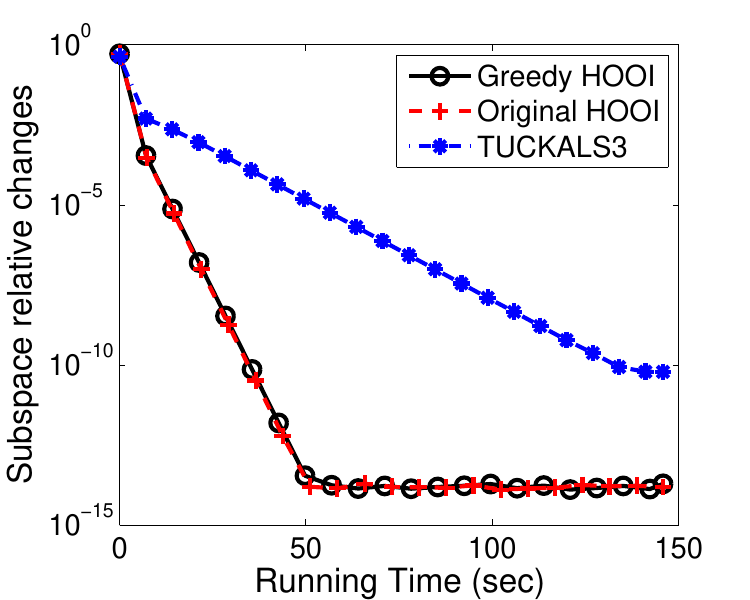}
\end{tabular}
\caption{Comparison of HOOIs and TUCKALS3 \cite{kroonenberg1980principal} on a randomly generated tensor of $50\times50\times50$ with core size $5\times5\times5$ and the Yale Face Database B \cite{georghiades2001few, lee2005acquiring} of size $38\times64\times2958$ with core size $5\times5\times20$. All three methods start from the same point, which is given by truncated HOSVD. The subspace relative change is calculated by $\frac{\sum_{n=1}^N\|\vA_n^k(\vA_n^k)^\top-\vA_n^{k+1}(\vA_n^{k+1})^\top\|_F}{\sum_{n=1}^N\|\vA_n^k(\vA_n^k)^\top\|_F}$, and it measures how far the current iterate deviates from satisfying the first-order optimality conditions. The results show that the original HOOI and the greedy HOOI {\color{blue}give the same relative change of multilinear subspace} at each iteration. {\color{blue}This is because they produce the same multilinear subspace.} They converge faster than TUCKALS3 on both synthetic data and the face image database.}
\label{fig:cvg-bh}
\end{figure}

\begin{figure}
\centering
\begin{tabular}{cc}
Synthetic data & Yale Face Database B\\
\includegraphics[width=0.35\textwidth]{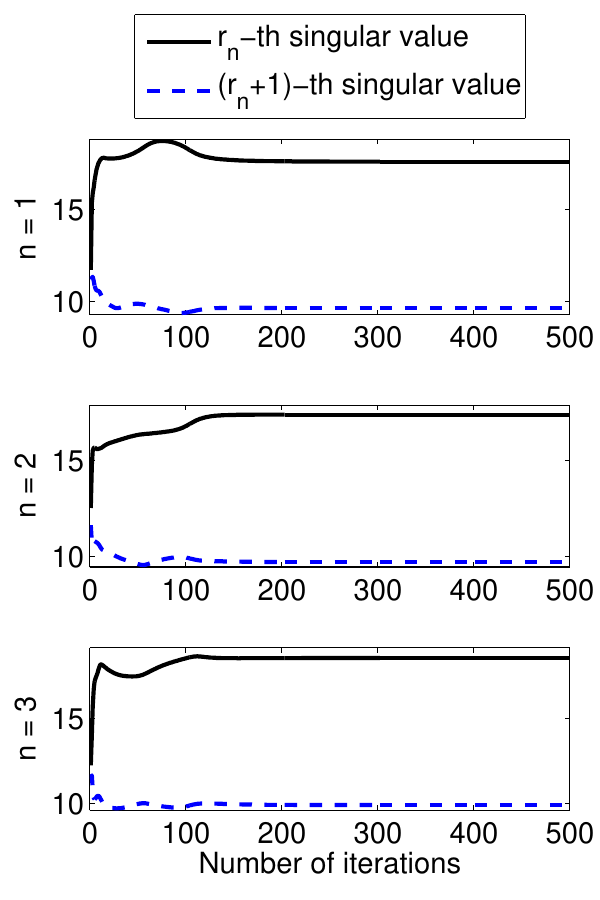} &
\includegraphics[width=0.35\textwidth]{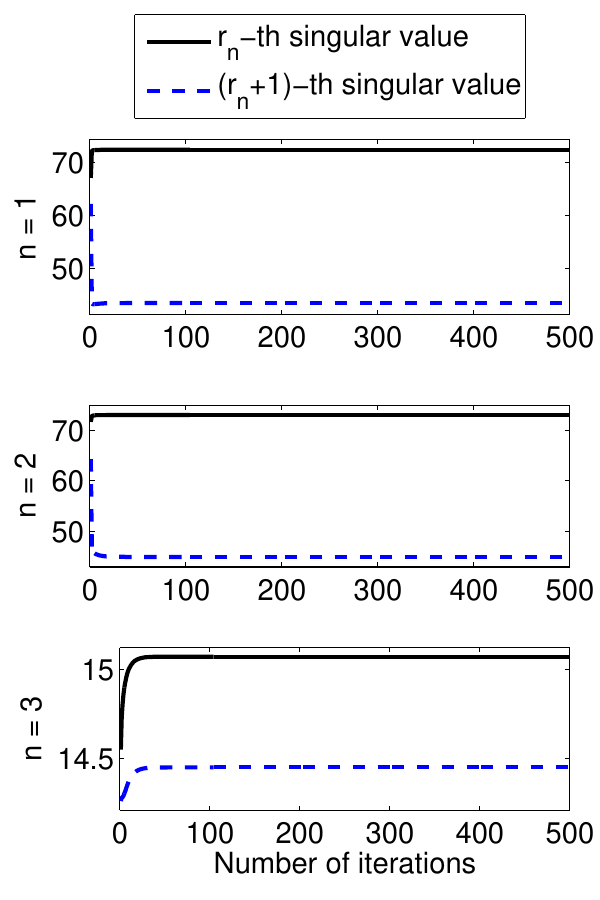}
\end{tabular}
\caption{More observations on the tests in Figure \ref{fig:cvg-bh}: the $r_n$-th and $(r_n+1)$-th singular values of $\vG_n^k$ generated by the original HOOI on the random dataset and the face image database; the values by the greedy HOOI are the same. Clearly, there are positive gaps between the $r_n$-th and $(r_n+1)$-th singular values of $\vG_n^k$ for all $n$ in the limit.}
\label{fig:svGap}
\end{figure}

\subsection{Contributions}

We summarize our contributions as follows.
\begin{itemize}
\renewcommand\labelitemi{--}
\item We propose a greedy HOOI method, which, for each update, selects from the best candidates one that is closest to the current iterate. With the greedy implementation, we show that any block-nondegenerate cluster point is a critical point and also a block-wise maximizer, and if a block-nondegenerate cluster point exists, then the entire iterate sequence converges to this cluster point.
\item Through relating the two iterate sequences by the original and greedy HOOIs, we --- for the first time --- establish global convergence of the original HOOI on multilinear subspace by assuming the existence of a block-nondegenerate cluster point and local convergence to a locally optimal subspace by assuming sufficient closeness of the starting point to a block-nondegenerate local maximizer. 
\item As a result, we show that the original HOOI converges to a globally optimal multilinear subspace, if the starting point is sufficiently close to any block-nondegenerate globally optimal solution.
\end{itemize}

\subsection{Notation and outline} We use bold capital letters $\vX,\vY,\ldots$ to denote matrices, caligraphic letters $\cS, \cU,\ldots$ for (set-valued) mappings, and bold caligraphic letters $\bm{\cX},\bm{\cY},\ldots$ for tensors. $\vI$ denotes an identity matrix, whose size is clear from the context. The $i$-th largest singular value of a matrix $\vX$ is denoted by $\sigma_i(\vX)$. The set of all orthonormal matrices in $\RR^{m\times r}$ is denoted as $\Stief_{m\times r}=\{\vX\in\RR^{m\times r}: \vX^\top\vX=\vI\}$. Throughout the paper, we focus on real field, but our analysis can be directly extended to complex field.

\begin{definition}[block-nondegeneracy]\label{def:nondeg}
A feasible solution $\vA$ of \eqref{eq:hosvd2} is block-nondegenerate if $\sigma_{r_n}(\vG_n)>\sigma_{r_n+1}(\vG_n),\,\forall n$, where 
\begin{equation}\label{eq:vgn}
\vG_n=\unfold_n(\bm{\cX}\times_{i\neq n}\vA_i^\top).
\end{equation}
\end{definition}

\begin{remark}\label{rm:non-deg}
In general, we are only able to claim convergence with existence of a block-nondegenerate cluster point. The original HOOI method can deviate from a critical point if it is not block-nondegenerate. To see this, suppose $\bar{\vA}$ is a block-wise maximizer and thus a critical point. Assume $\sigma_{r_1}(\bar{\vG}_1)=\sigma_{r_1+1}(\bar{\vG}_1)$. Let the original HOOI method start from $\bar{\vA}$ and update the first factor to $\tilde{\vA}_1$. Then $\tilde{\vA}_1$ may not span the same subspace as that by $\bar{\vA}_1$ because $\bar{\vG}_1$ has more than one dominant $r_1$-dimensional left singular subspaces. Therefore, we cannot guarantee the convergence of the learned multilinear subspace.
\end{remark}

The rest of the paper is organized as follows. 
In section \ref{sec:analysis}, global sequence convergence of the greedy HOOI is established under the assumption of the existence of a block-nondegenerate cluster point. The convergence of the original HOOI is shown in section \ref{sec:pf-of-main}. Finally, section \ref{sec:discussion} concludes the paper.

\section{Convergence analysis}\label{sec:analysis}
In this section, we assume the existence of one block-nondegenerate cluster point and establish global sequence convergence of Algorithm \ref{alg:ghooi} to a critical point. 
We first show some properties of the solution to \eqref{eq:a-sub}. These properties are important to show the block-wise maximality of a cluster point. 
Then we prove sufficient progress after each iteration of Algorithm \ref{alg:ghooi}. Finally we use the so-called Kurdyka-{\L}ojasiewicz (KL) property to establish the global sequence convergence. 
Note that if $\bar{\vA}$ is a critical point of \eqref{eq:hosvd2}, then letting $\bar{\bm{\cC}}=\bm{\cX}\times_1\bar{\vA}_1\ldots\times_N\bar{\vA}_N$, we have $(\bar{\bm{\cC}},\bar{\vA})$ to be a critical point of \eqref{eq:hosvd}. Hence, our analysis will only focus on \eqref{eq:hosvd2} and its equivalent form $\max_\vA F(\vA)$,
where
$$F(\vA)=f(\vA)-\sum_{n=1}^N g_n(\vA_n),$$
with $f(\vA) = \|\bm{\cX}\times_1\vA_1^\top\ldots\times_N\vA_N^\top\|_F^2$ and $$g_n(\vA_n)=\left\{
\begin{array}{ll}
0,&\text{ if }\vA_n\in\Stief_{I_n\times r_n},\\
+\infty, &\text{ otherwise. }
\end{array}\right.$$ 


Since $F$ is a semi-algebraic function, it has the so-called KL property (e.g., see \cite{xu2013block}), namely, for any point $\bar{\vA}$, in a neighborhood $\cN(\bar{\vA},\rho)$, there exists $\phi(s)=cs^{1-\theta}$ for some $c>0$ and $\theta\in[0,1)$ such that
\begin{equation}\label{eq:KL-F}
\phi'(|F(\vA)-F(\bar{\vA})|)\mathrm{dist}(\mathbf{0},\partial F(\vA))\ge 1,\text{ for any }\vA\in \cN(\bar{\vA},\rho) \text{ and }F(\vA)\neq F(\bar{\vA}).
\end{equation}

The KL property was introduced by {\L}ojasiewicz \cite{lojasiewicz1993geometrie} on real analytic functions. 
Kurdyka extended this property to 
functions on the $o$-minimal structure in \cite{kurdyka1998gradients}. Recently, the KL inequality was 
extended to nonsmooth sub-analytic functions \cite{bolte2007lojasiewicz}. The works \cite{attouch2010proximal, xu2013block} give a lot of concrete examples that own the property.

\subsection{First-order optimality conditions}
The Lagrangian function of \eqref{eq:hosvd2} is
$$\cL(\vA,\bm{\Lambda})=\frac{1}{2}\|\bm{\cX}\times_1\vA_1^\top\ldots\times_N\vA_N^\top\|_F^2-\frac{1}{2}\sum_{n=1}^N\langle\bm{\Lambda}_n,\vA_n^\top\vA_n-\vI\rangle,$$
where $\bm{\Lambda}=(\bm{\Lambda}_1,\ldots,\bm{\Lambda}_N)$ is the Lagrangian multiplier. The KKT conditions or first-order optimality conditions of \eqref{eq:hosvd2} can be derived from $\nabla \cL=\vzero$, namely, 
\begin{subequations}\label{eq:kkt0}
\begin{align}
&{\vG}_n{\vG}_n^\top\vA_n-\vA_n\bm{\Lambda}_n=\vzero,\,\forall n,\label{eq:kkt0-grad}\\
&\vA_n^\top\vA_n-\vI=\vzero,\,\forall n,\label{eq:kkt-feas}
\end{align}
\end{subequations}
where $\vG_n$ is defined in \eqref{eq:vgn}.
From \eqref{eq:kkt0}, we have $\bm{\Lambda}_n=\vA_n^\top{\vG}_n{\vG}_n^\top\vA_n$. Hence, the condition in \eqref{eq:kkt0-grad} can be written to
\begin{subequations}\label{eq:kkt}
\begin{align}
&{\vG}_n{\vG}_n^\top\vA_n=\vA_n\vA_n^\top{\vG}_n{\vG}_n^\top\vA_n,\,\forall n,\tag{\ref{eq:kkt0}c}\label{eq:kkt-grad}
\end{align}
\end{subequations}
\addtocounter{equation}{-1}
The above optimality conditions state
that the projection of every block-gradient to the tangent space of
the Stiefel manifold is zero. 
A point $\bar{\vA}$ is a critical point of \eqref{eq:hosvd2} if it satisfies the conditions in \eqref{eq:kkt-feas} and \eqref{eq:kkt-grad}.  



The following result is well known, and we will use it several times in our convergence analysis.
\begin{lemma}[von Neumann's Trace Inequality \cite{mirsky1975trace}]\label{lem:von-ineq}
For any matrices $\vX, \vY \in\RR^{m\times p}$, it holds that
\begin{equation}\label{eq:von-ineq}|\langle\vX,\vY\rangle|\le \sum_{i=1}^{\min(m,p)}\sigma_i(\vX)\sigma_i(\vY).
\end{equation}
The inequality \eqref{eq:von-ineq} holds with equality if $\vX$ and $\vY$ have the same left and right singular vectors.
\end{lemma}

\subsection{Properties of the solution to \eqref{eq:a-sub}}
To show the convergence of Algorithm \ref{alg:ghooi}, we analyze the solution of the subproblem \eqref{eq:a-sub}. The established properties are important to show the block-wise maximality of a cluster point. Problem \eqref{eq:a-sub} can be written in the following general form: 
\begin{equation}\label{eq:projh}
\min_{\vZ\in\cH_\vY}\|\vZ-\vX\|_F^2,
\end{equation}
where $\vX\in\Stief_{m\times r}$ and $\vY\in\RR^{m\times p}$ are given, and \begin{equation}\label{eq:chy}\cH_\vY=\argmax_{\vZ\in\Stief_{m\times r}}\|\vZ^\top\vY\|_F^2.\end{equation}

\begin{definition}[Quotient set of left leading singular vectors]
Given a matrix $\vY\in\RR^{m\times p}$ and positive integer $r\le \min(m,p)$, define $$\cB(\vY,r)=\{\vU\in\Stief_{m\times r}: \mathrm{span}(\vU) \text{ is a dominant }r\text{-dimensional left singular subspace of }\vY\}.$$
For any $\vU_1,\vU_2\in\cB(\vY,r)$, if $\mathrm{span}(\vU_1)=\mathrm{span}(\vU_2)$, i.e., they span the same subspace, we say they are equivalent. 
By this equivalence relation, we partition $\cB(\vY,r)$ to a set of equivalence classes and form a quotient set denoted as $\cU(\vY,r)$. 
\end{definition}

\begin{remark}
Throughout the paper, we regard $\cU(\vY,r)$ as the finite set of orthonormal matrices, and each of its elements is a representative of the bases that span the same subspace. If $\sigma_r(\vY)>\sigma_{r+1}(\vY)$, then $\vY$ has a unique dominant $r$-dimensional left singular subspace, and $\cU(\vY,r)$ is a singleton. Otherwise if $\sigma_r(\vY)=\sigma_{r+1}(\vY)$, then $\vY$ has multiple dominant $r$-dimensional left singular subspaces, and $\cU(\vY,r)$ has more than one element.
\end{remark}

\begin{proposition}\label{prop:projh} The problem \eqref{eq:projh} has a unique solution if and only if the following two conditions hold:
\begin{enumerate}
\item If $\vU_*\in\argmax_{\vU\in\cU(\vY,r)}\|\vU^\top\vX\|_*$, then $\vU_*^\top\vX$ is nonsingular;
\item For any $\vU\in\cU(\vY,r)$, if $\vU\neq\vU_*$, then $\|\vU^\top\vX\|_* < \|\vU_*^\top\vX\|_*$;
\end{enumerate}
where $\|\cdot\|_*$ denotes matrix nuclear norm, defined as the sum of all singular values of a matrix.
\end{proposition}
\begin{proof}
The ``only if'' part is easy to see, so we only prove the ``if'' part. Assume $\tilde{\vZ}$ and $\hat{\vZ}$ are both solutions of \eqref{eq:projh}. Note that $\cH_\vY$ in \eqref{eq:chy} is exactly the set $\cB(\vY,r)$. Hence, $\tilde{\vZ}=\vU_{\tilde{z}}\vW_{\tilde{z}}$ and $\hat{\vZ}=\vU_{\hat{z}}\vW_{\hat{z}}$ for $\vU_{\tilde{z}}, \vU_{\hat{z}}\in\cU(\vY,r)$ and some $\vW_{\tilde{z}},\vW_{\hat{z}}\in\Stief_{r\times r}$. Note 
$$\|\tilde{\vZ}-\vX\|_F^2=2r-2\langle\tilde{\vZ},\vX\rangle=2r-2\langle\vW_{\tilde{z}},\vU_{\tilde{z}}^\top\vX\rangle.$$
Then by Lemma \ref{lem:von-ineq} and the optimality of $\tilde{\vZ}$ on solving \eqref{eq:projh}, we have 
\begin{equation}\label{temp-eq1}\langle\vW_{\tilde{z}},\vU_{\tilde{z}}^\top\vX\rangle=\sum_{i=1}^r\sigma_i(\vU_{\tilde{z}}^\top\vX)=\|\vU_{\tilde{z}}^\top\vX\|_*=\max_{\vU\in\cU(\vY,r)}\|\vU^\top\vX\|_*.
\end{equation} Hence, from items 1 and 2, it follows that $\vU_{\tilde{z}}=\vU_*$, and similarly $\vU_{\hat{z}}=\vU_*$.

Let $\vU_*^\top\vX=\bar{\vU}\bar{\bm{\Sigma}}\bar{\vV}^\top$ be the full SVD of $\vU_*^\top\vX$ and $\vV_{\tilde{z}}=\vW_{\tilde{z}}^\top\bar{\vU}$, so $\vW_{\tilde{z}}=\bar{\vU}\vV_{\tilde{z}}^\top$. Then from \eqref{temp-eq1}, it holds that
$$\sum_{i=1}^r\sigma_i(\vU_*^\top\vX)=\langle\vW_{\tilde{z}},\vU_*^\top\vX\rangle=\langle\vV_{\tilde{z}}^\top,\bar{\bm{\Sigma}}\bar{\vV}^\top\rangle=\langle\vV_{\tilde{z}}^\top\bar{\vV},\bar{\bm{\Sigma}}\rangle=\sum_{i=1}^r\sigma_i(\vU_*^\top\vX)(\vV_{\tilde{z}}^\top\bar{\vV})_{ii}.$$
Note that $\sigma_i(\vU^\top\vX)>0$ and $(\vV_{\tilde{z}}^\top\bar{\vV})_{ii}\le 1$. The equality $\sum_{i=1}^r\sigma_i(\vU_*^\top\vX)=\sum_{i=1}^r\sigma_i(\vU_*^\top\vX)(\vV_{\tilde{z}}^\top\bar{\vV})_{ii}$ holds only if $(\vV_{\tilde{z}}^\top\bar{\vV})_{ii}=1$. Since $\vV_{\tilde{z}}^\top\bar{\vV}$ is orthogonal, we must have $\vV_{\tilde{z}}^\top\bar{\vV}=\vI$. Hence, $\vV_{\tilde{z}}=\bar{\vV}$ and $\vW_{\tilde{z}}=\bar{\vU}\bar{\vV}^\top$. For the same reason, $\vW_{\hat{z}}=\bar{\vU}\bar{\vV}^\top$. Therefore, $\tilde{\vZ}=\hat{\vZ}$, and the solution of \eqref{eq:projh} is unique.
\hfill\end{proof}


\begin{definition}[Unique-solution mapping]\label{def:ST}
Let
$$\cS(\vY,r)=\{\vX\in\Stief_{m\times r}: \vX \text{ satisfies the two conditions in Proposition \ref{prop:projh}}\}.$$
For any $\vX\in\cS(\vY,r)$, define $T_{\vY,r}(\vX)$ as the unique solution of \eqref{eq:projh}. 
\end{definition}

\begin{remark}\label{rm:sol-sub}
The proof of Proposition \ref{prop:projh} provides a way for finding a solution of \eqref{eq:projh}. Find $\vU_*\in\argmax_{\vU\in\cU(\vY,r)}\|\vU^\top\vX\|_*$ and get full SVD of $\vU_*^\top\vX=\bar{\vU}\bar{\bm{\Sigma}}\bar{\vV}^\top$. Then $\vZ_*=\vU_*\bar{\vU}\bar{\vV}^\top$ is a solution of \eqref{eq:projh}.
\end{remark}

Using Proposition \ref{prop:projh}, one can easily show the following two corollaries.
\begin{corollary}
If $\vX$ is sufficiently close to one $\vU$ in $\cB(\vY,r)$, then the solution of \eqref{eq:projh} is unique.
\end{corollary}

\begin{corollary}
If $\vX\in\cB(\vY,r)$, then $T_{\vY,r}(\vX) = \vX$, i.e., $\vX$ is a fixed point.
\end{corollary}

Furthermore, we can show the continuity of $T_{\vY,r}$.
\begin{theorem}\label{thm:opS}
The mapping $T_{\vY,r}$ is continuous on $\cS(\vY,r)$.
\end{theorem}

\begin{proof}
For convenience of the description, in this proof, we simply write $\cU(\vY,r), \cS(\vY,r)$ and $T_{\vY,r}$ to $\cU, \cS$ and $T$, respectively. 

For any $\vX\in\cS$, let $\vZ=T(\vX)$. If $T$ is not continuous at $\vX$, then there exists $\epsilon>0$ and a sequence $\{\vX^k\}_{k\ge1}$ in $\cS$ such that
$\|\vX-\vX^k\|_F\le \frac{1}{k}$ and $\|\vZ-\vZ^k\|_F\ge \epsilon$, where $\vZ^k=T(\vX^k)$. By the definition of $\cS$, we know that there is $\vU\in\cU$ such that $\|\vU^\top\vX\|_*>\|\tilde{\vU}^\top\vX\|_*$ for any $\tilde{\vU}\in\cU\backslash\{\vU\}$. Similarly, there is a sequence $\{\vU^k\}_{k\ge1}$ in $\cU$ such that for each $k$, $\|(\vU^k)^\top\vX^k\|_*>\|\tilde{\vU}^\top\vX^k\|_*$ for any $\tilde{\vU}\in\cU\backslash\{\vU^k\}$.

Let $\delta=\|\vU^\top\vX\|_*-\max_{\tilde{\vU}\in\cU\backslash\{\vU\}}\|\tilde{\vU}^\top\vX\|_*>0$. There is a sufficiently large integer $k_0$ such that for all $k\ge k_0$, it holds $\|\vU^\top\vX^k\|_*\ge \|\vU^\top\vX\|_*-\frac{\delta}{4}$ and $\|(\vU^k)^\top\vX^k\|_*\le \|(\vU^k)^\top\vX\|_*+\frac{\delta}{4}$. Note $\|\vU^\top\vX^k\|_*\le \|(\vU^k)^\top\vX^k\|_*$. Hence, $\|\vU^\top\vX\|_*-\frac{\delta}{4}\le \|(\vU^k)^\top\vX\|_*+\frac{\delta}{4}$, i.e., $\|\vU^\top\vX\|_*\le \|(\vU^k)^\top\vX\|_*+\frac{\delta}{2}$. Therefore, by the definition of $\delta$, it must hold that $\vU^k=\vU,\,\forall k\ge k_0$.

Hence, we can write $\vZ = \vU\vW_z$ and $\vZ^k=\vU\vW_{z^k}$  for all $k\ge k_0$, where $\vW_z, \vW_{z^k}\in\Stief_{r\times r}$. Note $\vU^\top\vX^k\to\vU^\top\vX$ as $k\to\infty$. Then from the proof of Proposition \ref{prop:projh}, we have $\vW_{z^k}\to \vW_z$ and thus $\vZ^k\to\vZ$ as $k\to\infty$. This contradicts to $\|\vZ-\vZ^k\|_F\ge \epsilon$. Therefore, $T$ is continuous at $\vX$. Since $\vX$ is an arbitrary point in $\cS$, this completes the proof.
\hfill\end{proof}

One can also show the following result. 
\begin{theorem}\label{thm:contY}
Assume $\sigma_r(\vY)>\sigma_{r+1}(\vY)$ and $\vY^k\to\vY$ as $k\to\infty$. If $\vX\in\cS(\vY,r)$, then there is a sufficiently large integer $k_0$ such that $\vX\in\cS(\vY^k,r)$ for all $k\ge k_0$, and
\begin{equation}\label{limT}\lim_{k\to\infty}T_{\vY^k,r}(\vX)=T_{\vY,r}(\vX).\end{equation}
\end{theorem}

\begin{proof}
By the assumption $\sigma_r(\vY)>\sigma_{r+1}(\vY)$, $\cU(\vY,r)$ is a singleton. Let $\vU\in\cU(\vY,r)$. Then from $\vX\in\cS({\vY,r)}$, it follows that $\vX^\top\vU$ is nonsingular. Since  $\vY^k\to\vY$ as $k\to\infty$, there exists an integer $k_0$, such that $\sigma_r(\vY^k)>\sigma_{r+1}(\vY^k)$, i.e., $\cU(\vY^k,p)$ is a singleton for all $k\ge k_0$. Let $\vU^k\in\cU(\vY^k,r),\,\forall k$. We can choose the representative satisfying $\vU^k\to\vU$, since $\vY^k\to\vY$. Therefore, taking another larger $k_0$ if necessary, we have that $\vX^\top\vU^k$ is nonsingular and thus $\vX\in\cS(\vY^k,r)$ for all $k\ge k_0$. Finally, using Remark \ref{rm:sol-sub} and $\vU^k\to\vU$, we have \eqref{limT} and complete the proof.  
\hfill\end{proof}
%
We also need the following result, which will be used to show the criticality and block-wise maximality of a cluster point of the sequence given by Algorithm \ref{alg:ghooi}.
\begin{lemma}\label{lem:crit-pt}
For any feasible solution $\bar{\vA}$, if $T_{\bar{\vG}_n,r_n}(\bar{\vA}_n)=\bar{\vA}_n,\,\forall n$, then $\bar{\vA}$ is a critical point and also a block-wise maximizer of \eqref{eq:hosvd2}, where 
\begin{equation}\label{eq:barvgn}\bar{\vG}_n=\unfold_n(\bm{\cX}\times_{i\neq n}\bar{\vA}_i).
\end{equation}
\end{lemma}

\begin{proof}
Note that $T_{\bar{\vG}_n,r_n}(\bar{\vA}_n)=\bar{\vA}_n,\,\forall n$ implies that $\bar{\vA}_n$ is a basis of the dominant $r_n$-dimensional left singular subspace of $\bar{\vG}_n$. Hence, $\bar{\vA}_n\bar{\vA}_n^\top\bar{\vG}_n\bar{\vG}_n^\top\bar{\vA}_n
=\bar{\vG}_n\bar{\vG}_n^\top\bar{\vA}_n,\,\forall n$. Therefore, $\bar{\vA}$ is a critical point. 

In addition, $T_{\bar{\vG}_n,r_n}(\bar{\vA}_n)=\bar{\vA}_n,\,\forall n$ implies that $\bar{\vA}_n$ is a solution to $\max_{\vA_n}\|\vA_n^\top\bar{\vG}_n\|_F^2$ over $\Stief_{I_n\times r_n}$ for all $n$. Hence, $\bar{\vA}$ is a block-wise maximizer. This completes the proof.
\hfill\end{proof}

%
%
%
%
%



\subsection{Bounding iterate distance by objective progress}

As shown below, for the problem \eqref{eq:projh}, if there is a positive gap between $\sigma_r(\vY)$ and $\sigma_{r+1}(\vY)$, 
the distance between $\vX$ and $\vZ$ 
can be bounded by the objective difference.

\begin{theorem}\label{thm:key-ineq}
Given $\vX\in\Stief_{m\times r}$ and $\vY\in\RR^{m\times p}$, any solution $\vZ$ of \eqref{eq:projh} satisfies 
\begin{equation}\label{eq:key-ineq}\frac{\sigma_r^2(\vY)-\sigma_{r+1}^2(\vY)}{2}\|\vZ-\vX\|_F^2\le \|\vZ^\top\vY\|_F^2-\|\vX^\top\vY\|_F^2.
\end{equation}
\end{theorem}
\begin{proof}
Note $\vZ=\vU\vW_z$ for some $\vU\in\cU(\vY,r)$ and $\vW_z\in\Stief_{r\times r}$.
Let $\vY=\vU\bm{\Sigma}\vV^\top+
\vU_\perp\bm{\Sigma}_\perp\vV_\perp^\top$ be the full SVD of $\vY$.
Also, let 
$\vW=\vU^\top\vX$ and $\vW_\perp=\vU_\perp^\top\vX.$
Then $\vX=\vU\vW+\vU_\perp\vW_\perp$ and $\vW^\top\vW+\vW_\perp^\top\vW_\perp=\vI$ from $\vX^\top\vX=\vI$.

As in the proof of Proposition \ref{prop:projh}, we have 
\begin{equation}\label{key-ineq-t1}
\|\vZ^\top\vY\|_F^2=\sum_{i=1}^r\sigma_i^2(\vY)
\end{equation} and 
\begin{equation}\label{key-ineq-t2}
\|\vZ-\vX\|_F^2=2r-2\langle \vZ,\vX\rangle=2r-2\langle \vW_z,\vW\rangle=2r-2\sum_{i=1}^r\sigma_i(\vW),
\end{equation}
where the last equality is from Lemma \ref{lem:von-ineq} and the optimality of $\vZ$ for \eqref{eq:projh}. Also, note that
\begin{equation}\label{key-ineq-t3}\|\vX^\top\vY\|_F^2=\|\vW^\top\bm{\Sigma}\|_F^2+\|\vW_\perp^\top\bm{\Sigma}_\perp\|_F^2.
\end{equation}

Assume $\vW_\perp=\tilde{\vU}\tilde{\bm{\Sigma}}\tilde{\vV}^\top$ to be the full SVD of $\vW_\perp$. Then
$$\vW^\top\vW=\vI-\vW_\perp^\top\vW_\perp=\tilde{\vV}(\vI-\tilde{\bm{\Sigma}}^\top\tilde{\bm{\Sigma}})\tilde{\vV}^\top.$$
Let $\tilde{\sigma}_1\ge\tilde{\sigma}_2\ge\ldots\ge\tilde{\sigma}_r$ be the first $r$ largest singular values of $\vW_\perp$. Then $\sigma_i(\vW)=\sqrt{1-\tilde{\sigma}_{r-i+1}^2},\,\forall i$, and using Lemma \ref{lem:von-ineq} again, we have
\begin{equation}\label{key-ineq-t4}\|\vW^\top\bm{\Sigma}\|_F^2=\langle\vW\vW^\top,\bm{\Sigma}^2\rangle\le\sum_{i=1}^r(1-\tilde{\sigma}_i^2)\sigma_{r-i+1}^2(\vY),
\end{equation}
and
\begin{equation}\label{key-ineq-t5}\|\vW_\perp^\top\bm{\Sigma}_\perp\|_F^2=\langle\vW_\perp\vW_\perp^\top,\bm{\Sigma}_\perp\bm{\Sigma}_\perp^\top\rangle\le
\sum_{i=1}^r\tilde{\sigma}_i^2\sigma_{r+i}^2(\vY).
\end{equation}
Hence, from \eqref{key-ineq-t1} and \eqref{key-ineq-t3} through \eqref{key-ineq-t5}, we have
\begin{align}\label{key-ineq-t6}
\|\vZ^\top\vY\|_F^2-\|\vX^\top\vY\|_F^2=&\sum_{i=1}^r\sigma_i^2(\vY)-\|\vW^\top\bm{\Sigma}\|_F^2-\|\vW_\perp^\top\bm{\Sigma}_\perp\|_F^2\cr
\ge&\sum_{i=1}^r\tilde{\sigma}_i^2\big(\sigma_{r-i+1}^2(\vY)-\sigma_{r+i}^2(\vY)\big)\cr
\ge &\sum_{i=1}^r\tilde{\sigma}_i^2\big(\sigma_{r}^2(\vY)-\sigma_{r+1}^2(\vY)\big),\end{align}
where the last inequality is from $\sigma_r^2(\vY)-\sigma_{r+1}^2(\vY)\le \sigma_{r-i+1}^2(\vY)-\sigma_{r+i}^2(\vY),\,\forall i$.
Using the fact $1-\sqrt{1-x}\le x,\,\forall x\in [0, 1]$, we have
$$2r-2\sum_{i=1}^r\sigma_i(\vW)=2r-2\sum_{i=1}^r\sqrt{1-\tilde{\sigma}_i^2}
\le 2\sum_{i=1}^r\tilde{\sigma}_i^2,$$
and thus from \eqref{key-ineq-t2}, it follows that
$$\|\vZ-\vX\|_F^2\le2\sum_{i=1}^r\tilde{\sigma}_i^2.$$
Plugging the above inequality into \eqref{key-ineq-t6}, we have the desired result.\hfill
\end{proof}

Using Theorem \ref{thm:key-ineq}, we show the following result of sufficient progress.

\begin{lemma}[Sufficient progress]\label{lem:sq-bd}
Let $\{\vA^k\}_{k\ge1}$ be the sequence generated from Algorithm \ref{alg:ghooi}. Assume it has a block-nondegenerate cluster point $\bar{\vA}$. Then there is a constant $\alpha$ such that if $\vA^k$ is sufficiently close to $\bar{\vA}$, we have 
$$\alpha\|\vA^{k+1}-\vA^k\|_F^2\le F(\vA^{k+1})-F(\vA^k).$$
\end{lemma}

\begin{proof}
It is easy to see that there exists a small positive number $\delta$ such that if $\|\vA-\bar{\vA}\|_F\le\delta$, then
$$\sigma_{r_n}(\vG_n)-\sigma_{r_n+1}(\vG_n)\ge\frac{1}{2}(\sigma_{r_n}(\bar{\vG}_n)-\sigma_{r_n+1}(\bar{\vG}_n))\triangleq\alpha_n>0,\,\forall n,$$
where the strict inequality is from the block-nondegeneracy of $\bar{\vA}$. Assume $\vA^k$ is sufficiently close to $\bar{\vA}$ such that
$$\sum_{n=1}^N\sqrt{\frac{2(F(\bar{\vA})-F(\vA^k))}{\alpha_n}}+\|\vA^k-\bar{\vA}\|_F\le \delta.$$ From Theorem \ref{thm:key-ineq}, it follows that
$$\frac{\alpha_1}{2}\|\vA_1^{k+1}-\vA_1^k\|_F^2\le \|(\vA_1^{k+1})^\top\vG_1^k\|_F^2-\|(\vA_1^k)^\top\vG_1^k\|_F^2\le F(\bar{\vA})-F(\vA^k),$$
where $\vG_1^k$ is defined in \eqref{eq:gnk}, and we have used \eqref{temp-eq2}.
Hence, $\|\vA_1^{k+1}-\vA_1^k\|_F\le \sqrt{\frac{2(F(\bar{\vA})-F(\vA^k))}{\alpha_1}}$ and
$$\|(\vA_1^{k+1},\vA_{>1}^k)-\bar{\vA}\|_F\le \|\vA_1^{k+1}-\vA_1^k\|_F+\|\vA^k-\bar{\vA}\|_F\le \delta.$$
Repeating the above arguments, in general, we have for all $n$ that
$$\|\vA_n^{k+1}-\vA_n^k\|_F\le \sqrt{\frac{2(F(\bar{\vA})-F(\vA^k))}{\alpha_n}},$$
and
$$\|(\vA_{\le n}^{k+1},\vA_{>n}^k)-\bar{\vA}\|_F\le \sum_{i=1}^n\|\vA_i^{k+1}-\vA_i^k\|_F+\|\vA^k-\bar{\vA}\|_F\le \delta.$$
Therefore, every intermediate point $(\vA_{\le n}^{k+1},\vA_{>n}^k)$ is in $\cN(\bar{\vA},\delta)\triangleq\{\vA: \|\vA-\bar{\vA}\|_F\le \delta\}$, and thus for all $n$,
$$\frac{\alpha_n}{2}\|\vA_n^{k+1}-\vA_n^k\|_F^2\le \|(\vA_n^{k+1})^\top\vG_n^k\|_F^2-\|(\vA_n^k)^\top\vG_n^k\|_F^2.$$
Let $\alpha=\min_n \frac{\alpha_n}{2}>0$. Summing the above inequality from $n=1$ to $N$ gives the desired result. \hfill
\end{proof}

\subsection{Global sequence convergence result}
Using Lemma \ref{lem:sq-bd} and the KL property of $F$, we show the global sequence convergence of Algorithm \ref{alg:ghooi} to a critical point. Our proof follows two steps. In the first step, we show criticality of any cluster point; in the second step, we 
apply the claim made in \cite{bolte2014proximal}: for problem $\max_\vx \Phi(\vx)$, if the sequence $\{\vx^k\}_{k\ge1}$ generated by a certain algorithm satisfies the following two properties
\begin{enumerate}
\item \emph{Sufficient progress}: there is a constant $\rho_1>0$, such that $\rho_1\|\vx^{k+1}-\vx^k\|^2\le \Phi(\vx^{k+1})-\Phi(\vx^k),\,\forall k$;
\item \emph{Subgradient lower bound}: there is a constant $\rho_2>0$ such that for any $k$, for some $\vg^{k+1}\in\partial \Phi(\vx^{k+1})$, it holds $\|\vg^{k+1}\|\le \rho_2\|\vx^{k+1}-\vx^k\|,$
\end{enumerate}
then the KL property of $\Phi$ implies that $\{\vx^k\}$ is a Cauchy sequence.


\begin{theorem}[Global sequence convergence]\label{thm:glb-cvg}Let $\{\vA^k\}_{k\ge1}$ be the sequence generated from Algorithm \ref{alg:ghooi}. 
If $\{\vA^k\}_{k\ge1}$ has a block-nondegenerate cluster point $\bar{\vA}$, then 
$\bar{\vA}$ is a critical point and block-wise maximizer of \eqref{eq:hosvd2}, and 
\begin{equation}\label{eq:cvg}\lim_{k\to\infty}\vA^k=\bar{\vA}.\end{equation}
\end{theorem}


\begin{proof}
We first show the criticality and block-wise maximality of $\bar{\vA}$.
Suppose that $\bar{\vA}$ is one block-nondegenerate 
Since $\bar{\vA}$ is a cluster point, there is a subsequence $\{\vA^k\}_{k\in\cK}$ convergent to $\bar{\vA}$. From the update rule in \eqref{eq:a-sub}, it is easy to see
\begin{equation}\label{temp-eq2}
\|(\vA_n^{k+1})^\top\vG_n^k\|_F^2\le \|\bar{\vA}_n^\top\bar{\vG}_n\|_F^2,\,\forall k,\,\forall n.
\end{equation}
We claim that $\bar{\vA}_1$ is a solution of $\max_{\vA_1\in\Stief_{I_1\times r_1}}\|\vA_1^\top\bar{\vG}_1\|_F^2$. Otherwise, $\|\bar{\vA}_1^\top\bar{\vG}_1\|_F^2<\sum_{i=1}^{r_1}\sigma_i^2(\bar{\vG}_1)$. Note $$\lim_{\cK\ni k\to\infty}\|(\vA_1^{k+1})^\top\vG_1^k\|_F^2=\lim_{\cK\ni k\to\infty}\sum_{i=1}^{r_1}\sigma_i^2(\vG_1^k)
=\sum_{i=1}^{r_1}\sigma_i^2(\bar{\vG}_1),$$ which contradicts to \eqref{temp-eq2}. Hence, $T_{\bar{\vG}_1,r_1}(\bar{\vA}_1)=\bar{\vA}_1$.

Note that $\vG_1^k\to\bar{\vG}_1$ as $\cK\ni k\to\infty$ and $\vA_1^k\in\cS(\vG_1^k,r_1)$ as $k\in\cK$ is sufficiently large. From the block-nondegeneracy of $\bar{\vA}$ and Theorems \ref{thm:opS} and \ref{thm:contY}, we have \begin{equation}\label{temp-eq3}\lim_{\cK\ni k\to\infty}\vA_1^{k+1}=\lim_{\cK\ni k\to\infty}T_{\vG_1^k,r_1}(\vA_1^k)= T_{\bar{\vG}_1,r_1}(\bar{\vA}_1)=\bar{\vA}_1.
\end{equation} Hence, taking a sufficiently large $k\in\cK$, we can make $\|\vA_1^{k+1}-\vA_1^k\|_F$ sufficiently small, and thus we can repeat the above arguments for $n=2,\ldots,N$ to conclude $$\bar{\vA}_n\in\argmax_{\vA_n\in\Stief_{I_n\times r_n}}\|\vA_n^\top\bar{\vG}_n\|_F^2,\,\forall n.$$ Therefore, from the definition of $T_{\bar{\vG}_n,r_n}$, it holds that $T_{\bar{\vG}_n,r_n}(\bar{\vA}_n)=\bar{\vA}_n,\,\forall n$, and $\bar{\vA}$ is a critical point and a block-wise maximizer of \eqref{eq:hosvd2} from Lemma \ref{lem:crit-pt}.


%
Note that there is a constant $L$ such that 
\begin{equation}\label{eq:lip-f}
\|\nabla_{\vA_n} f(\tilde{\vA})-\nabla_{\vA_n}f(\hat{\vA})\|_F\le L\|\tilde{\vA}-\hat{\vA}\|_F,\,\forall \tilde{\vA},\hat{\vA}\in\cO,\,\forall n,
\end{equation}
where $$\cO=\{\vA:\,\vA=(\vA_1,\ldots,\vA_N),\, \vA_n\in\Stief_{I_n\times r_n},\,\forall n\}.$$
%
%
For any $n=1,\ldots,N$, from the optimality of $\vA_n^k$ on problem $\max_{\vA_n}F(\vA_{i<n}^k,\vA_n,\vA_{i>n}^{k-1})$, it holds that
\begin{align*}
&\,\vzero\in\partial_{\vA_n} F(\vA_{i\le n}^k,\vA_{i>n}^{k-1})\\
\Leftrightarrow &\, \vzero\in\nabla_{\vA_n}f(\vA_{i\le n}^k,\vA_{i>n}^{k-1})+\partial g_n(\vA_n^k)\\
\Leftrightarrow &\, \nabla_{\vA_n}f(\vA^k)-\nabla_{\vA_n}f(\vA_{i\le n}^k,\vA_{i>n}^{k-1})\in\nabla_{\vA_n}f(\vA^k)+\partial g_n(\vA_n^k).
\end{align*}
Hence,
\begin{align}\label{eq:temp1}
\dist(\vzero, \partial F(\vA^k))\le & \sum_{n=1}^N\|\nabla_{\vA_n}f(\vA^k)-\nabla_{\vA_n}f(\vA_{i\le n}^k,\vA_{i>n}^{k-1})\|_F\cr
\overset{\eqref{eq:lip-f}}{\le} & NL\|\vA^k-\vA^{k-1}\|_F,
\end{align}
which together with Lemma \ref{lem:sq-bd} indicates that $\{\vA^k\}_{k\ge1}$ owns the two properties stated in the beginning of this subsection. In addition, $F$ has the KL-property in \eqref{eq:KL-F}, and thus $\{\vA^k\}_{k\ge1}$ is a Cauchy sequence and converges. Since $\bar{\vA}$ is a cluster point, then $\vA^k\to\bar{\vA}$ as $k\to\infty$.  
%
This completes the proof.\hfill
\end{proof}

\begin{remark}
The result in \eqref{temp-eq3} is a key step to have the criticality and block-wise maximality. In general, without the block-nondegeneracy assumption, it may not hold.
\end{remark}

As long as the starting point is sufficiently close to any block-nondegenerate local maximizer, Algorithm \ref{alg:ghooi} will yield an iterate sequence convergent to a local maximizer as summarized below.
\begin{theorem}[Convergence to local minimizer]\label{thm:loc-min}
Assume Algorithm \ref{alg:ghooi} starts from any point $\vA^0$ that is sufficiently close to one block-nondegenerate local maximizer $\vA^*$ of $F(\vA)$. Then the sequence $\{\vA^k\}_{k\ge1}$ converges to a local maximizer.
\end{theorem}

\begin{proof}
First, note that if some $\vA^{k_0}$ is sufficiently close to $\vA^*$ and $F(\vA^{k_0})=F(\vA^*)$, then $\vA^{k_0}$ must also be a local maximizer and block-nondegenerate. In this case, $\vA^k=\vA^{k_0},\,\forall k\ge k_0$. Hence, without loss of generality, we can assume $F(\vA^k)<F(\vA^*),\,\forall k$. Secondly, note that in the proof of Theorem \ref{thm:glb-cvg}, we only use $F(\vA^k)<F(\bar{\vA})$ and the sufficient closeness of $\vA^0$ to $\bar{\vA}$ to show $\{\vA^k\}_{k\ge1}$ to be a Cauchy sequence. Therefore, repeating the same arguments, we can show that if $\vA^0$ is sufficiently close to $\vA^*$, then $\{\vA^k\}_{k\ge1}$ is a Cauchy sequence and thus converges to a block-nondegenerate point $\bar{\vA}$ near $\vA^*$. From Theorem \ref{thm:glb-cvg}, it follows that $\bar{\vA}$ is a critical point. We claim $F(\bar{\vA})=F(\vA^*)$, i.e., $\bar{\vA}$ is a local maximizer. If otherwise $F(\bar{\vA})<F(\vA^*)$, then by the KL inequality, it holds that $\phi'(F(\vA^*)-F(\bar{\vA}))\mathrm{dist}(\vzero,\partial F(\bar{\vA}))\ge 1$, which contradicts to $\vzero\in\partial F(\bar{\vA})$. Hence, $F(\bar{\vA})=F(\vA^*)$. This completes the proof.
\hfill\end{proof}

From Theorem \ref{thm:loc-min}, we can easily get the following local convergence to a globally optimal solution.
\begin{theorem}[Global optimality]\label{thm:glb-opt}
Assume Algorithm \ref{alg:ghooi} starts from any point $\vA^0$ that is sufficiently close to one block-nondegenerate globally optimal solution $\vA^*$ of \eqref{eq:hosvd2}. Then the sequence $\{\vA^k\}_{k\ge1}$ converges to a globally optimal solution.
\end{theorem}

\section{Proof of the main theorem}\label{sec:pf-of-main} In this section, we analyze the convergence of the original HOOI method by relating its iterate sequence to that of the greedy HOOI method. Because any solution to each subproblem of the original HOOI method is still a solution after arbitrary rotation, we do not hope to establish convergence on the iterate sequence $\{\vA^k\}_{k\ge1}$ itself. Instead, we show the convergence of the projection matrix sequence $\{\vA^k(\vA^k)^\top\}_{k\ge1}$.

First note that
\begin{equation}\label{eq:equiv-obj}\|\bm{\cX}\times_1\vA_1^\top\ldots\times_N\vA_N^\top\|_F^2=\left\langle\bm{\cX},\bm{\cX}\times_1(\vA_1\vA_1^\top)\ldots\times(\vA_N\vA_N^\top)\right\rangle.
\end{equation}
We also need the following two lemmas.
\begin{lemma}\label{limit-pt}
If $\bar{\vA}\bar{\vA}^\top=\tilde{\vA}\tilde{\vA}^\top$ and $\tilde{\vA}$ is a critical point of \eqref{eq:hosvd2}, then $\bar{\vA}$ is also a critical point.
\end{lemma}
\begin{proof}
Since $\tilde{\vA}$ is a critical point of \eqref{eq:hosvd2}, it holds that $\tilde{\vG}_n\tilde{\vG}_n^\top\tilde{\vA}_n=
\tilde{\vA}_n\tilde{\vA}_n^\top\tilde{\vG}_n\tilde{\vG}_n^\top\tilde{\vA}_n$ and $\tilde{\vA}_n^\top \tilde{\vA}_n=\vI$ for all $n$. Note that $\bar{\vA}\bar{\vA}^\top=\tilde{\vA}\tilde{\vA}^\top$ implies $\bar{\vG}_n\bar{\vG}_n^\top=\tilde{\vG}_n\tilde{\vG}_n^\top$. Hence, for any $n$,
$$\bar{\vG}_n\bar{\vG}_n^\top\bar{\vA}_n\bar{\vA}_n^\top
=\tilde{\vG}_n\tilde{\vG}_n^\top\tilde{\vA}_n\tilde{\vA}_n^\top=
\tilde{\vA}_n\tilde{\vA}_n^\top\tilde{\vG}_n\tilde{\vG}_n^\top\tilde{\vA}_n\tilde{\vA}_n^\top
=\bar{\vA}_n\bar{\vA}_n^\top\bar{\vG}_n\bar{\vG}_n^\top\bar{\vA}_n\bar{\vA}_n^\top.$$
Multiplying $\bar{\vA}_n$ to both sides and noting $\bar{\vA}_n^\top\bar{\vA}_n=\vI$ gives
$$\bar{\vG}_n\bar{\vG}_n^\top\bar{\vA}_n=\bar{\vA}_n\bar{\vA}_n^\top\bar{\vG}_n\bar{\vG}_n^\top\bar{\vA}_n,\,\forall n,$$
and thus $\bar{\vA}$ is a critical point.
\hfill\end{proof}

\begin{lemma}\label{lem:nchg}
Let $\{\vA^k\}_{k\ge 1}$ be the sequence generated by the original HOOI method and assume it has a block-nondegenerate cluster point $\bar{\vA}$. If for some $k_0$, $F(\vA^{k_0})=F(\bar{\vA})$, then there is an integer $K\ge k_0$ such that $\vA^k(\vA^k)^\top=\vA^K(\vA^K)^\top,\,\forall k\ge K$.
\end{lemma}

\begin{proof}
Because $F(\vA^k)$ is nondecreasing and upper bounded,  we have $\lim_{k\to\infty}F(\vA^k)=F(\bar{\vA})$ and $F(\vA^k)\le F(\bar{\vA})$, so if $F(\vA^{k_0})=F(\bar{\vA})$, then $F(\vA^k)=F(\bar{\vA}),\,\forall k\ge k_0$. 

Since $\bar{\vA}$ is a cluster point, there must be an integer $K\ge k_0$ such that $\vA^K$ is sufficiently close to $\bar{\vA}$ and $\vA^K$ is block-nondegenerate. Hence, $\vG_1^K$ has a unique dominant $r_1$-dimensional left singular subspace. Note $$\max_{\vA_1\in\Stief_{I_1\times r_1}}\|\vA_1^\top\vG_1^K\|_F^2=\sum_{i=1}^{r_1}\sigma_i^2(\vG_1^K)=F(\bar{\vA})=\|(\vA_1^K)^\top\vG_1^K\|_F^2.$$
Therefore, $\vA_1^{K}$ and $\vA_1^{K+1}$ both span the dominant $r_1$-dimensional left singular subspace of $\vG_1^K$, and thus $\vA_1^{K+1}(\vA_1^{K+1})^\top=\vA_1^{K}(\vA_1^{K})^\top$. Using \eqref{eq:equiv-obj}, we can repeat the arguments to have $\vA_n^{K+1}(\vA_n^{K+1})^\top=\vA_n^{K}(\vA_n^{K})^\top,\,\forall n$, i.e., $\vA^{K+1}=\vA^K$. Now starting from $\vA^{K+1}$ and repeating the arguments, we have the desired result.
\hfill\end{proof}

By Lemma \ref{lem:nchg}, without loss of generality, we assume $F(\vA^k)<F(\bar{\vA}),\,\forall k$ in the remaining analysis. With Lemmas \ref{limit-pt} and \ref{lem:nchg}, we are now ready to prove the main theorem.

\begin{proof}[Proof of Theorem \ref{thm:main}]  

\textbf{Part (i):} Since $\bar{\vA}$ is a cluster point of $\{\vA^k\}$, there is a subsequence $\{\vA^k\}_{k\in\cK}$ convergent to $\bar{\vA}$, and there is $k_0\in\cK$ such that $\vA^{k_0}$ is sufficiently close to $\bar{\vA}$. Without loss of generality, we assume that $\vA^0$ is sufficiently close to $\bar{\vA}$ because otherwise we can set $\vA^{k_0}$ as a new starting point and the convergence of $\{\vA^k\}_{k\ge 1}$ is equivalent to that of $\{\vA^k\}_{k\ge k_0}$. Let $\{\tilde{\vA}^k\}$ be the sequence generated by the greedy HOOI method starting from $\tilde{\vA}^0=\vA^0$. We go to show that if $\vA^0$ is sufficiently close to $\bar{\vA}$, then \begin{equation}\label{equal-iter}
\vA^k(\vA^k)^\top=\tilde{\vA}^k(\tilde{\vA}^k)^\top,\,\forall k\ge 1.
\end{equation} 

Repeating the same arguments in the proof of Lemma \ref{lem:sq-bd}, we have that if $\vA^0$ is sufficiently close to $\bar{\vA}$, then $\tilde{\vA}^1$ is also sufficiently close to $\bar{\vA}$. Note that when $\vA^0$ is sufficiently close to $\bar{\vA}$, it is block-nondegenerate and $\sigma_{r_1}(\vG_1^0)>\sigma_{r_1+1}(\vG_1^0)$. Hence, $\vA_1^1$ and $\tilde{\vA}_1^1$ both span the dominant $r_1$-dimensional left singular subspace of ${\vG}_1^0$ and thus $\vA^1_1(\vA^1_1)^\top=\tilde{\vA}^1_1(\tilde{\vA}^1_1)^\top$. Since both $\tilde{\vA}^0$ and $\tilde{\vA}^1$ are sufficiently close to $\bar{\vA}$, we have $\sigma_{r_2}(\tilde{\vG}_2^0)>\sigma_{r_2+1}(\tilde{\vG}_2^0)$. Note $\vG_2^0(\vG_2^0)^\top=\tilde{\vG}_2^0(\tilde{\vG}_2^0)^\top$. Hence, $\vA_2^1$ and $\tilde{\vA}_2^1$ both span the dominant $r_2$-dimensional left singular subspace of $\vG_2^0(\vG_2^0)^\top$ and thus $\vA^1_2(\vA^1_2)^\top=\tilde{\vA}^1_2(\tilde{\vA}^1_2)^\top$. Repeating the above arguments, we have $\vA^1_n(\vA^1_n)^\top=\tilde{\vA}^1_n(\tilde{\vA}^1_n)^\top,\,\forall n$, i.e., $\vA^1(\vA^1)^\top=\tilde{\vA}^1(\tilde{\vA}^1)^\top$.

Assume that for some integer $K\ge1$, it holds $\vA^k(\vA^k)^\top=\tilde{\vA}^k(\tilde{\vA}^k)^\top$ and $\tilde{\vA}^k\in\cN(\bar{\vA},\rho)$ for all $k\le K$, where $\rho$ is sufficiently small and plays the same role as that in the proof of Theorem \ref{thm:glb-cvg}. From \eqref{eq:equiv-obj}, it follows that $F(\tilde{\vA}^k)=F(\vA^k)<F(\bar{\vA}),\,\forall k\ge K$. Through the same arguments as those in the proof of Theorem \ref{thm:glb-cvg}, we have $\tilde{\vA}^{K+1}\in\cN(\bar{\vA},\rho)$, and thus $\vA^{K+1}(\vA^{K+1})^\top=\tilde{\vA}^{K+1}(\tilde{\vA}^{K+1})^\top$ by the above arguments that show $\vA^1(\vA^1)^\top=\tilde{\vA}^1(\tilde{\vA}^1)^\top$. By induction, we have the result in \eqref{equal-iter}. 

Taking another subsequence if necessary, we can assume $\{\tilde{\vA}^k\}_{k\in\cK}$ converging to $\tilde{\vA}$ and thus $\bar{\vA}\bar{\vA}^\top=\tilde{\vA}\tilde{\vA}^\top$ by \eqref{equal-iter}. Note that the block-nondegeneracy of $\tilde{\vA}$ is equivalent to that of $\bar{\vA}$. Hence, $\tilde{\vA}$ is block-nondegenerate and is a critical point and a block-wise maximizer, and $\tilde{\vA}^k$ converges to $\tilde{\vA}$ by Theorem \ref{thm:glb-cvg}. Therefore, $\vA^k(\vA^k)^\top$ converges to $\bar{\vA}\bar{\vA}^\top$. From Lemma \ref{limit-pt}, we have that $\bar{\vA}$ is a critical point of \eqref{eq:hosvd2}, and from \eqref{eq:equiv-obj}, $\bar{\vA}$ is a block-wise maximizer. This completes the proof of part (i).

\textbf{Part (ii):} Let $\{\tilde{\vA}^k\}_{k\ge1}$ be the sequence generated by the greedy HOOI method starting from $\tilde{\vA}^0=\vA^0$. From Theorem \ref{thm:loc-min}, it follows that $\tilde{\vA}^k$ converges to a local maximizer $\tilde{\vA}$ of \eqref{eq:hosvd2}. In addition, by similar arguments as those in the proof of part (i), we can show that \eqref{equal-iter} still holds. Hence, $\vA^k(\vA^k)^\top$ converges to $\tilde{\vA}\tilde{\vA}^\top$, and this completes the proof.
\hfill\end{proof}

\section{Conclusions}\label{sec:discussion}
We proposed a greedy HOOI method and established its iterate sequence convergence by assuming existence of a block-nondegenerate cluster point. Through relating the iterates by the original HOOI to those by the greedy HOOI, we have shown the global convergence of the HOOI method on multilinear subspace sequence. In addition, if the starting point is sufficiently close to any block-nondegenerate locally optimal point, we showed that the original HOOI could guarantee convergence to a locally optimal multilinear subspace. 

\bibliographystyle{gLMA}
\bibliography{optimization,nmfc,paperCol}
\end{document}